\documentclass[10pt, a4paper,reqno]{amsart}

\usepackage{amsmath}
\usepackage{amsfonts}
\usepackage{amssymb}
\usepackage{amsthm}
\usepackage[active]{srcltx}
\usepackage{color}
\usepackage[colorlinks,citecolor=blue,linkcolor=blue]{hyperref}

\numberwithin{equation}{section}

\newtheorem{theorem}[equation]{Theorem}
\newtheorem{lemma}[equation]{Lemma}

\newtheorem{corollary}[equation]{Corollary}

\theoremstyle{definition}
\newtheorem{definition}[equation]{Definition}

\theoremstyle{remark}
\newtheorem{remark}[equation]{Remark}

\def\kint_#1{\mathchoice%
          {\mathop{\kern 0.2em\vrule width 0.6em height 0.69678ex depth -0.58065ex
                  \kern -0.8em \intop}\nolimits_{\kern -0.4em#1}}%
          {\mathop{\kern 0.1em\vrule width 0.5em height 0.69678ex depth -0.60387ex
                  \kern -0.6em \intop}\nolimits_{#1}}%
          {\mathop{\kern 0.1em\vrule width 0.5em height 0.69678ex depth -0.60387ex
                  \kern -0.6em \intop}\nolimits_{#1}}%
          {\mathop{\kern 0.1em\vrule width 0.5em height 0.69678ex depth -0.60387ex
                  \kern -0.6em \intop}\nolimits_{#1}}}
\def\vintslides_#1{\mathchoice%
          {\mathop{\kern 0.1em\vrule width 0.5em height 0.697ex depth -0.581ex
                  \kern -0.6em \intop}\nolimits_{\kern -0.4em#1}}%
          {\mathop{\kern 0.1em\vrule width 0.3em height 0.697ex depth -0.604ex
                  \kern -0.4em \intop}\nolimits_{#1}}%
          {\mathop{\kern 0.1em\vrule width 0.3em height 0.697ex depth -0.604ex
                  \kern -0.4em \intop}\nolimits_{#1}}%
          {\mathop{\kern 0.1em\vrule width 0.3em height 0.697ex depth -0.604ex
                  \kern -0.4em \intop}\nolimits_{#1}}}

\newcommand{\df}[1]{\buildrel\mbox{\small def}\over{#1}}

\newcommand{\eps}{\varepsilon}

\newcommand{\R}{\mathbb{R}}
\newcommand{\Rn}{\mathbb{R}^d}

\newcommand{\esssup}{\operatornamewithlimits{ess\, sup}}

\newcommand{\spt}{\operatorname{spt}}

\newcommand{\divt}{\operatorname{div}}

\renewcommand{\l}{\left}
\renewcommand{\r}{\right}

\def\Xint#1{\mathchoice
{\XXint\displaystyle\textstyle{#1}}%
{\XXint\textstyle\scriptstyle{#1}}%
{\XXint\scriptstyle\scriptscriptstyle{#1}}%
{\XXint\scriptscriptstyle\scriptscriptstyle{#1}}%
\!\int}

\def\XXint#1#2#3{{\setbox0=\hbox{$#1{#2#3}{\int}$}
\vcenter{\hbox{$#2#3$}}\kern-.5\wd0}}

\def\dashint{\Xint-}

\title[On the interior regularity of the 2-D Euler equations]{On the interior regularity of weak solutions to the 2-D incompressible Euler equations}

\author[J.~Siljander]{Juhana Siljander}
\address{University of Jyv\"askyl\"a, Department of Mathematics and Statistics, P.O. Box 35, FI-40014 University of Jyv\"askyl\"a, Finland}{}
\email{juhana.siljander@jyu.fi}

\author[J.M.~Urbano]{Jos\'{e} Miguel Urbano}
\address{CMUC, Department of
Mathematics, University of Coimbra, 3001-501 Coimbra, Portugal}{}
\email{jmurb@mat.uc.pt}

\begin{document}



\begin{abstract}
We consider the 2-D incompressible Euler equations in a bounded domain and show that local weak solutions are exponentially integrable, uniformly in time, under minimal integrability conditions. This is a Serrin-type interior regularity result
\begin{equation*}
u \in L_{\rm loc}^{2+\eps}(\Omega_T) \implies {\rm regularity}
\end{equation*}
for weak solutions in the energy space $L_t^\infty L_x^2$ satisfying appropriate vorticity estimates. The argument is completely local in nature as the result follows from the structural properties of the equation alone, while completely avoiding all sorts of boundary conditions and related gradient estimates. To the best of our knowledge, the approach we follow is new in the context of Euler equations and provides an alternative look at interior regularity issues. We also show how our method can be used to give a modified proof of the classical Serrin condition for the regularity of the Navier-Stokes equations in any dimension.
\end{abstract}

\maketitle

\section{Introduction}

In this paper, we consider local weak solutions $u \in L_{\rm loc}^2(\Omega_T)$ of the incompressible Euler equations in two space dimensions
\begin{equation}\label{equation}
\partial_t u + (u \cdot \nabla)u + \nabla p = 0 \qquad \textrm{and} \qquad   {\rm div}\, u = 0,
\end{equation}
and show, under appropriate integrability conditions on $u$ and its vorticity $\omega$, they are exponentially integrable uniformly in time.

Despite their physical relevance, for instance in the context of Kolmogorov's theory of turbulence, much of what we know about the weak solutions of \eqref{equation} concerns some sort of pathological behaviour. For example, weak solutions are not unique:  Scheffer constructed in \cite{Sche93} a two-dimensional weak solution which is compactly supported in space and time; a simpler counter-example was provided later by Shnirelman  \cite{Shni97}. More recently, De Lellis and Szekelyhidi Jr. \cite{DeLeSzek09} have shown that even locally bounded weak solutions may fail to be unique. Another line of research deals with Onsager's conjecture that the critical exponent for H\"older continuous weak solutions to dissipate energy is $\tfrac{1}{3}$. The current threshold exponent for anomalous dissipation is $\frac{1}{5}$ (see \cite{Iset12, BuckDeLeIsetSzek13}) and the fact that any $C^{0,\frac{1}{3}}$ weak solution conserves energy was proven by Constantin, E and Titi in \cite{ConsETiti94}. For more on Euler equations we refer to the book of Majda and Bertozzi \cite{BerMaj}, and the expository articles by Constantin \cite{Cons07} or Bardos and Titi \cite{BardTiti07}. Recent interesting papers in two dimensions are \cite{KisSve, KisZla}.

Instead of studying pathological behaviour, we aim here at establishing a regularity result. In the whole space $\R^2$, or in a bounded domain with the Neumann boundary condition $u \cdot \nu =0$, it is well known that sufficient integrability of the vorticity $\omega$ implies the gradient estimate
\begin{equation}\label{grad_est}
u, \omega \in L_t^\infty L_x^2 \implies u \in L_t^\infty H_x^1.
\end{equation}
By Poincar\'e inequality, one immediately obtains that $u \in L_t^\infty BMO_x$, which is also known to be optimal by explicit solution formulas. Our contribution is in showing that almost the same optimal regularity can be obtained, by purely local arguments, merely from the structural properties of equations~\eqref{equation}. 

More precisely, we work in a domain $\Omega \subset \R^2$, without any sort of boundary conditions, and prove the exponential integrability of the velocity, uniformly in time, under minimal integrability assumptions. Our main result is a Serrin-type~\cite{Serr62} interior regularity theorem
\begin{equation*}
u \in L_{\rm loc}^{2+\eps}(\Omega_T) \implies {\rm regularity}
\end{equation*}
for weak solutions in the energy space $L_t^\infty L_x^2$. In particular, we do not make any kind of assumptions on initial data.

In the proof we study the vorticity formulation of the Euler equations. In order to make sense of weak solutions, we assume $\omega \in L_{\rm loc}^2(\Omega_T)$ and exploit the corresponding energy estimates for the vorticity. We then combine these {\it vorticity estimates} with the Hardy-Littlewood-Sobolev theorem on fractional integration, and apply a Biot-Savart type potential estimate to improve the integrability of the velocity field in a Moser iteration procedure, to obtain the main result. Our reasoning, in a nutshell, is that a proper modification of the Moser local regularity machinery (see \cite{Mose61}) can still be applied completely without gradient estimates, such as~\eqref{grad_est}, if we have a suitable Biot-Savart type result at our disposal. 

The same approach naturally applies to the Navier-Stokes system and we give a modified proof of the classical Serrin regularity condition~\cite{Serr62} through a Moser iteration technique where, contrary to the original argument of Serrin, we first prove $u \in L_t^\infty L_x^p$, for a large enough $p$, and only then show the boundedness of the vorticity $\omega$. While this machinery is probably well-known to the experts, the argument is still somewhat delicate and all the details we provide are scattered and difficult to find in the literature. Interestingly enough, we manage to improve the original Serrin result in the two dimensional setting.

\section{Preliminaries and Main result} \label{section:preliminaries}

\subsection{Notation}
We denote by $B_r(x_o)$ the standard open ball in $\Rn$, with radius $r>0$ and centre at $x_o$. Let $\Omega$ be an open set in $\Rn$ and denote the space-time cylinder $\Omega_T\df=\Omega \times (0, T]$. For the boundary of the set we use the standard notation $\partial \Omega$. We define the parabolic boundary of a space-time cylinder as
\[
\partial_p \Omega_T \df= \big(\partial \Omega \times [0, T] \big) \cup \big(\Omega \times \{0\}\big)
\]
and denote $d\mu \df=dx\,dt$.

As we will only consider interior regularity results, it is enough to consider local space-time cylinders $B_r(x_o) \times (t_o - r, t_o)$. By translation, we may always assume $x_o=t_o=0$. For such cylinder, we will use the notation
\[
Q_r  =B_r \times I_r \df = B_r(0) \times (-r,0),
\]
for $r>0$. For the integral average, we write
\[
\dashint_\Omega |u| \, dx\df=\frac1{|\Omega|}\int_\Omega |u| \, dx.
\]
The corresponding averaged $L^q$-norm will be denoted by
\[
\|u\|_{L^q(\Omega), avg}\df= \l(\dashint_\Omega |u|^q \, dx\r)^{\frac1q},
\]
for $q> 0$. For $\gamma >0$, we define the class of exponentially integrable functions $u \in \rm{exp\big(\hspace{-.1truecm}\sqrt[\gamma]{L(\Omega)}\big)}$ by requiring
\[
\int_\Omega e^{\sqrt[\gamma]{|u(x)|}} \, dx < \infty.
\]
The corresponding local space is defined by requiring the finiteness of this integral for all compact sets $K \Subset \Omega$.

\bigskip

\subsection{Definitions} We now make precise the concept of weak solution we will be using.

\begin{definition}\label{weak_solution}
We say a vector field $u \in L_{\rm loc}^2(\Omega_T)$ is a weak solution of equation~\eqref{equation} in $\Omega_T$ if 
\[
\int_{\Omega} u(x, t) \cdot  \nabla \psi(x) \, dx = 0,
\]
for every $\psi\in H_0^{1}(\Omega)$ and almost every $t \in [0, T]$, and if
\begin{equation}\label{weak_sol_eq}
\begin{split}
&\int_{0}^{T}\int_{\Omega} u \cdot \partial_t\phi   + \sum_{i=1}^d u_i(u\cdot \nabla) \phi_i   \, dx \, dt =  0,
\end{split}
\end{equation}
for every smooth test function $\phi$ with compact support in space-time such that $\nabla \cdot \phi=0$.

\end{definition}

Let us now restrict to the two-dimensional setting and introduce the vorticity
$$\omega \df= {\rm curl}\, u = \partial_{x_1} u_2 -  \partial_{x_2} u_1 .$$
It is an easy exercise to show that if $u$ is a weak solution of equation~\eqref{equation} in $\Omega_T$ and $\omega \in L_{\rm loc}^2(\Omega_T)$, then the vorticity equation
\[
\partial_t \omega + (u\cdot \nabla) \omega= 0
\]
holds in the weak sense.
We will exploit in the sequel energy estimates for this equation, which we formalize into the following definition.

\begin{definition}\label{def:VE}
Let $\Omega \subset \R^2$ and $u$ be a weak solution of equation~\eqref{equation} in $\Omega_T$. We say that $u$ satisfies the \textit{vorticity estimates} if $\omega \in L_{\rm loc}^2(\Omega_T)$ and there exists a constant $V_o >0$ such that
\begin{equation}\label{VE}\tag{VE}
\begin{split}
&\esssup_{0 \le t \le T} \int_{\Omega} |\omega|^{1+\alpha}\zeta(x,t) \, dx \\
&\le V_o \int_{0}^{T}\int_{\Omega} |\omega|^{1+\alpha}\big[ |u| |\nabla \zeta|+\left| \partial_t \zeta \right| \big]  \, d\mu,
\end{split}
\end{equation}
for every $ \alpha \in [0,1)$, and for every non-negative test function $\zeta \in C^\infty(\Omega_T)$ that vanishes in a neighbourhood of the parabolic boundary $\partial_p\Omega_T$.
\end{definition}

\begin{remark}
A rigorous treatment of the vorticity estimates would typically require an approximation scheme, for example in the spirit of the celebrated DiPerna-Lions theory \cite{DiPL}. Strictly speaking, this requires the vorticity to be at least in $L_t^\infty L_x^{2}$, which is stronger than we assume. This notwithstanding, subsequent results by  Ambrosio~\cite{Ambr04} and Bouchut-Crippa~\cite{BC13} brought the integrability threshold closer to ours. 

The focus in this paper is not in deriving these \textit{a priori} estimates but rather in showing that, once they hold, interior regularity can be obtained under minimal integrability assumptions. The underlying reasoning parallels the use of De Giorgi classes in the context of regularity theory for second-order elliptic and parabolic equations \cite{DeGi57, GianVesp06}; see also \cite{Vass07} for an application to the Navier-Stokes system. 

\end{remark}

\subsection{The main result} Our main theorem is an exponential integrability result implying that the velocity field $u$ of the incompressible Euler system cannot form singularities too fast if it is -- a priori -- sufficiently integrable together with its vorticity. This is a Serrin-type~\cite{Serr62} interior regularity result yielding, for weak solutions $u \in L_t^\infty L_x^2$ satisfying \eqref{VE},
\begin{equation*}
u \in L_{\rm loc}^{2+\eps}(\Omega_T) \implies {\rm regularity}.
\end{equation*}
In other words, if a weak solution blows up then either our integrability assumptions must be violated or the blow-up can only be, roughly speaking, of logarithmic type.

\begin{theorem}\label{main_theorem}
Let $u  \in L_{\rm loc}^\infty(0, T; L_{\rm loc}^{2}(\Omega)) $ be a weak solution of equation~\eqref{equation} in two dimensions satisfying the vorticity estimates~\eqref{VE}. Then, if $u \in    L_{\rm loc}^{2+\eps}(\Omega_T)$ for some $\eps>0$, there exists a constant $C=C(\eps, V_o)>0$ such that  for every $Q_{4r}=B_{4r} \times I_{4r} \Subset \Omega_T$ we have
\[
u \in L^\infty(I_r; \rm{exp\big(\hspace{-.1truecm}\sqrt[\gamma]{L(B_r)}\big)}),
\]
where
\[
\gamma  \df= C \l[\l(r^2\dashint_{Q_{4r}} |\omega|^2 \, d\mu+1\r)^{\frac12}\l(\dashint_{Q_{4r}} |u|^{2+\eps} \ d\mu+1\r)^{\frac1{2+\eps}}+\|u\|_{L^\infty(I_{4r};L^1(B_{4r}), avg)}\r].
\]

\end{theorem}

\begin{remark}
If $\Omega = \R^2$, or in the case of the Neumann boundary condition $u \cdot \nu =0$, the DiPerna-Lions borderline assumption is known to be equivalent to $u \in L_t^\infty H_x^{1}$, provided $u$ belongs to the energy space $L_t^\infty L_x^{2}$. In particular, we have
\begin{equation}\label{grad_embed}
u, \omega \in L_t^\infty L_x^2 \implies u \in L_t^\infty H_x^1
\end{equation}
and obtain immediately, from the Poincar\'e inequality, that $u \in L_t^\infty BMO_x$. This is also known to be optimal due to an explicit stationary solution obtained via 
\begin{equation*}
u(x) = \left(-\frac{x_2}{r^2}, \frac{x_1}{r^2}\right)^\mathrm{T} \int_0^r s\omega(s) \, ds, \quad r=|x|, 
\end{equation*}
with the radial vorticity
$$\omega(s) = \frac1{s \log \tfrac 1{s}}.$$
For details we refer to~\cite[Section 2.2.1]{BerMaj}. 

We work in the interior and our results are local in the sense that we do not use any sort of boundary (or initial) information, which is crucial in the above reasoning. Our contribution is thus in showing that one is still able to obtain the almost optimal interior integrability estimate $u \in L_t^\infty \exp(\hspace{-.1truecm}\sqrt[\gamma]{L_x})$ merely from the structural properties of equation~\eqref{equation}, while completely avoiding boundary conditions and the heavy machinery of gradient estimates such as~\eqref{grad_embed}. Finally, while we have stated the theorem with the standard energy space $L_t^\infty L_x^2$, our proof actually works with merely $u \in L_t^\infty L_x^1$.
\end{remark}

The proof uses a Moser iteration technique, with the Sobolev embedding replaced by the Hardy-Littlewood-Sobolev theorem on fractional integration. This provides the required bound for the Biot-Savart potential, which is then used to improve the integrability of $u$. We iterate the obtained estimates to deduce a quantitative growth rate for the spatial $L^p$-norm of $u$, uniformly in time. This enables us to show that the norms are exponentially summable, and we conclude the result.

In the case of the Navier-Stokes system, where we have additional control on the gradient, we may use our method to give a modified proof of the classical Serrin regularity condition. Moreover, we can relax the Serrin condition in two dimensions to assume merely $u \in L^q$ for $q>2$ instead of the Serrin condition which, more or less, assumes $q>4$.

\subsection{Auxiliary tools}

Although we do not employ the Sobolev embedding in our arguments for the Euler equations, we will use it in the very end when studying the Navier-Stokes system. Even there, we only use it as a formal tool, whereas the rigorous argument is conducted for suitable mollifications such that no existence of Sobolev gradients is required. We recall that for $d \ge 2$ and for $u \in H_0^{1}(\Omega)$, by the Sobolev embedding, there exists a constant $C=C(d)>0$ such that
\[
\l(\dashint_{\Omega} |u|^{2^*} \, dx\r)^{\frac1{2^*}} \le C\text{diam}(\Omega)\l(\dashint_{\Omega} |\nabla u|^{2} \, dx\r)^{\frac1{2}},
\]
where 
\[
2^*\df=
\begin{cases}
\text{any number in \,} (2,\infty), \quad &\text{for} \ d=2\\
\frac{2d}{d-2}, \quad &\text{for}\ d \ge 3.
\end{cases}
\]
An easy consequence of the above result is a parabolic Sobolev embedding for $u \in L^\infty(0, T; L^2(\Omega)) \cap L^2(0, T; H_0^{1}(\Omega))$. In this case, there exists a constant $C=C(d)>0$ such that
\begin{equation}\label{parabolic_Sobolev}
\begin{split}
&\l(\int_{0}^{T} \l( \dashint_{\Omega} |u|^q \, dx\r)^{\frac sq} \, dt \r)^{\frac1s} \\
&\qquad\le C\text{diam}(\Omega)^{\frac2s} \l(\esssup_{0 \le t \le T} \dashint_{\Omega} |u|^2 \, dx + \int_{0}^{T}\dashint_{\Omega} |\nabla u|^2 \, d\mu\r)^{\frac12},
\end{split}
\end{equation}
for any $2<q, s < \infty$ satisfying
\[
\frac dq +\frac2s =\frac d2.
\]

Next we turn to a Hardy-Littlewood-Sobolev lemma on fractional integration. We recall the definition of the Riesz potential
\[
(I_\beta f)(x)= (-\Delta)^{-\frac\beta2}f(x)= c_\beta \int_{\Rn} \frac{f(y)}{|x-y|^{d-\beta}} \, dy,
\]
with $0<\beta<d$, which arises naturally via the Biot-Savart law. Here the vorticity $\omega$ plays the role of $f$ and we have for the velocity field $|u| \le C |(I_1 \omega)(x)|$ (cf. Lemma~\ref{vorticity_singular}). We will however work in a ball rather than the whole space and for this reason we will always consider all locally defined functions as being extended as zero outside the local set to make sense of the above non-local integral. With this convention in mind, we have the following well-known lemma.

\begin{lemma}\label{singular}
Let $0<\beta<d$ and $f \in L^q(\Rn)$ for $1 < q < d/\beta$. Then there exists a constant $C=C(d, \beta, q)>0$ such that 
\[
\|I_{\beta} f\|_{L^s(\Rn)} \le C\|f\|_{L^q(\Rn)},
\]
for 
\[
s=\frac{dq}{d-\beta q}
\]
and
\[
C(d, \beta, q) \le C(d, \beta) \max\l\{(q-1)^{-\l(1-\frac\beta d\r)}, s^{1-\frac\beta d}\r\}.
\]
\end{lemma}
\noindent For the proof see ~\cite[Chapter V, \textsection 1.3]{Stei70}.

\section{A local Biot-Savart estimate}

In this section we establish a version of the local Biot-Savart law for solenoidal vector fields. Observe that, under our assumptions, the error term below is a constant, bounded both in space and time, unlike what happens in the classical paper of Serrin~\cite[Lemma 2]{Serr62}, where the error is known to be merely a harmonic function in space. In order to obtain the uniform bound in time, we need to use the assumption that $u$ is a weak solution and locally in $ L_t^\infty L_x^1\cap L_{t,x}^{2+\eps}$. In particular, it does not seem to be enough to assume $u$ being solenoidal. 

Let $\sigma >0$. In order to simplify the notation we will denote 
\[
\sigma Q \df= Q_{\sigma r}
\]
as well as
\[
\sigma I \df= I_{\sigma r}\quad \text{
and}
\quad
\sigma B \df= B_{\sigma r}.
\]
In the same spirit, for $\sigma=1$, we denote $Q\df=Q_r, B\df=B_r$ and $I\df= I_r$.

\begin{lemma}\label{vorticity_singular}
Let $\eps > 0$, $k \in \R^2$ and $0< \sigma < 1$. Suppose $u$ is a weak solution to equation~\eqref{equation} in $\Omega_T \Supset 2Q$ satisfying~\eqref{VE} such that $u \in L^{2+\eps}(2Q)$. Then there exists a constant $C=C(\eps, V_o)>0$ such that for every $(x,t) \in \sigma Q$ we have
\[
|u(x,t)-k|\le \frac1{2\pi}\int_{B} \frac{|\omega(y,t)|}{|x-y|}  \, dy + \frac{C}{(1-\sigma)^3}\big[\|u(\cdot, t)-k\|_{L^1(B), avg}+\Gamma\big],
\]
where
\begin{equation}\label{gamma_lemma}
\Gamma=\Gamma(2Q)\df=2r\l(\dashint_{2Q} |\omega|^2 \, d\mu\r)^{\frac12}\l(\dashint_{2Q} |u|^{2+\eps}+1 \ d\mu\r)^{\frac1{2+\eps}}.
\end{equation}
\begin{remark}
Observe that the term $\|u(\cdot, t)-k\|_{L^1(B), avg}$ is, a priori, well-defined only for almost every $t \in 2I$. In case this quantity is not well-defined, we interpret the value being infinite, and the estimate holds trivially.
\end{remark}

\begin{proof}
By translation, we may assume $Q$ to be centred at the origin, as already suggested by the slight abuse of notation in the statement of the lemma.
Since $\divt u=0$ in $\Omega$, there exists a stream function $\varphi(x, t)$ such that 
\begin{equation}\label{stream}
 u_1=-\partial_{x_2}\varphi \quad \text{and}\quad u_2=\partial_{x_1}\varphi.
\end{equation}
Observe that $ \omega = \text{curl\,}  u = \Delta \varphi$ in $\Omega \Supset B$. Recalling that the Green function for the Laplacian in the ball $B_r(0) \subset \R^2$ is given by
\[
G(x,y)=\frac{1}{2\pi}\l[\ln\l(\frac{|x|}{r}\l|y-\frac{x}{|x|^2}r^2\r|\r)-\ln|x-y|\r],
\]
we have 
\begin{equation}\label{representation}
\begin{split}
\varphi(x,t)&=-\int_{\partial B_r(0)} \frac{\partial G(x,y)}{\partial_y \nu}\varphi(y,t) \, d\mathcal H^{1}(y)- \int_{B_r(0)}G(x,y)\omega(y,t) \, dy \\
&\df= J_1(x,t)+J_2(x,t).
\end{split}\end{equation}
Here $\mathcal H^1$ denotes the one-dimensional Hausdorff measure on $\partial B_r(0)$. A tedious but straightforward calculation shows for $x, y \in B_r(0)$ that
\begin{align*}
\partial_{x_i} G(x,y) &= -\frac1{2\pi}\l[\frac{x_i-y_i}{|x-y|^2}-\frac{|y|\l(|y|x_i-\frac{y_i}{|y|}r^2\r)}{\l||y|x-\frac{y}{|y|}r^2\r|^2}\r] \\
&\le -\frac1{2\pi}\l[\frac{x_i-y_i}{|x-y|^2}-\frac{|y|}{\l||y|x-\frac{y}{|y|}r^2\r|}\r] \\
&\le -\frac1{2\pi}\l[\frac{x_i-y_i}{|x-y|^2}-\frac{1}{r-|x|}\r].
\end{align*}
Therefore, we obtain
\begin{equation}\notag
|\partial_{x_i}J_2(x,t)| \le \frac{1}{2\pi}\int_{B}\l[\frac{1}{|x-y|}+\frac{1}{r-|x|}\r]|\omega(y,t)| \, dy.
\end{equation}

We choose a test function $\zeta \in C^\infty(2Q)$ such that $\zeta =1$ in $Q$, $\zeta = 0$ on $\partial_p (2Q)$ and satisfies the bounds
\[
|\nabla \zeta| \le Cr^{-1} \quad \text{and} \quad \l|\frac{\partial\zeta}{\partial t}\r| \le Cr^{-1}.
\]
We apply the vorticity estimates~\eqref{VE} with $\alpha=0$ and H\"older's inequality to obtain
\begin{equation}\label{omega_eps}
\begin{split}
\esssup_{t \in I} \int_{B} |\omega(y,t)| \, dy &\le Cr^{-1}\int_{2Q}|\omega||u|+|\omega| \, d\mu \\
&\le Cr^{-1} \l(\int_{2Q} |\omega|^2 \, d\mu\r)^{\frac {1}{2}} \l(\int_{2Q} (|u|+1)^2 \, d\mu\r)^{\frac1{2}},
\end{split}
\end{equation}
and, therefore, we may conclude that
\begin{align*}
 \esssup_{\sigma B}\int_{B} \frac{|\omega(y,t)|}{r-|x|} \, dy &\le \frac{1}{(1-\sigma) r} \esssup_{t \in I}\int_{B} |\omega(y,t)| \, dy \\
  &\le \frac {C}{(1-\sigma)r^2} \l(\int_{2Q} |\omega|^2 \, d\mu\r)^{\frac1{2}}\l(\int_{2Q} (|u|+1)^2 \, d\mu\r)^{\frac1{2}} \\
    &\le \frac {Cr}{(1-\sigma)} \l(\dashint_{2Q} |\omega|^2 \, d\mu\r)^{\frac1{2}}\l(\dashint_{2Q} |u|^{2+\eps} \, d\mu+1\r)^{\frac1{2(1+\eps)}},
\end{align*}
uniformly for almost every $t \in I$.

Let $l\df=(-k_1, k_2)$. We have shown that 
\begin{equation}\label{law}
\partial_{x_i}\varphi(x,t)-l_j=[\partial_{x_i}J_1(x,t)-l_j]+\partial_{x_i}J_2(x,t), \quad i=1,2, i \neq j, 
\end{equation}
where
\begin{equation}\label{A_2}
\esssup_{\sigma B}|\partial_{x_i}J_2(x,t)| \le \frac1{2\pi}\int_{B} \frac{|\omega(y,t)|}{|x-y|}  \, dy + \frac{C\Gamma}{1-\sigma},
\end{equation}
for a constant $C$ independent of the time variable $t$. 

It remains to show a similar estimate for $\partial_{x_i}J_1(x,t)-l_j$. We begin by observing that, for each $t$, this is a harmonic function in the space variables. Therefore, using the $L^\infty - L^1$ estimate for harmonic functions, the representation~\eqref{law} and~\eqref{A_2}, as well as Jensen's inequality, give
\begin{equation}\label{harmonic_bound}
\begin{split}
&\esssup_{\sigma B} |\partial_{x_i} J_1(\cdot,t)-l_j|\\
&\le \frac{C}{(1-\sigma)^2}\dashint_{\frac{1+\sigma}2B} |\partial_{x_i} J_1(x,t)-l_j|\, dx \\
&\le \frac{C}{(1-\sigma)^2}\dashint_{\frac{1+\sigma}2B} |\partial_{x_i} \varphi(x,t)-l_j| \, dx +  \frac{C}{(1-\sigma)^2}\dashint_{\frac{1+\sigma}2B} |\partial_{x_i}J_2(x,t)| \, dx \\
&\le \frac{C}{(1-\sigma)^2}\dashint_{B} |u(x,t)-k| \, dx + \frac{C\Gamma}{(1-\sigma)^3}  \\
&\qquad+ \frac{C}{(1-\sigma)^2} \l(\dashint_{B}\l( \int_{B}\frac{|\omega(y,t)|}{|x-y|} \, dy \r)^{2(1+\eps)} \, dx\r)^{\frac1{2(1+\eps)}},
\end{split}
\end{equation}
for almost every $t \in I$.
Now, by the Hardy-Littlewood-Sobolev estimate of Lemma~\ref{singular} and by the vorticity estimates~\eqref{VE}, with $\alpha=2-\frac2{2+\eps}-1$, we obtain
that 
\begin{align*}
&\esssup_{t \in I}\l\|\int_{B}\frac{|\omega(y,t)|}{|\cdot-y|} \, dy\r\|_{L^{2(1+\eps)}(B), avg} \\
&\le C(\eps)\,r\esssup_{t \in I}\|\omega(\cdot,t)\|_{L^{2-\frac2{2+\eps}}(B), avg} \\
&\le Cr \l(\dashint_{2Q} |\omega|^2 \, d\mu\r)^{\frac12}\l(\dashint_{2Q} |u|^{2+\eps} \, d\mu+1\r)^{\frac{1}{2(1+\eps)}}.
\end{align*}
Combining the above estimates yields
\begin{equation}\notag
\begin{split}
&\esssup_{\sigma B} |\partial_{x_i} J_1(\cdot,t)-l_j| \\
&\le \frac{Cr}{(1-\sigma)^{3}}\l(\dashint_{2Q} |\omega|^2\ d\mu\r)^{1/2}\l( \dashint_{2Q} |u|^{2+\eps} +1\ d\mu\r)^{\frac1{2(1+\eps)}}+\frac{C\|u(\cdot, t)-k\|_{L^1(B), avg}}{(1-\sigma)^{2}},
\end{split}
\end{equation}
as required.
This, together with~\eqref{law},~\eqref{A_2} and~\eqref{stream} concludes the proof.
\end{proof}
\end{lemma}

As a simple corollary of the previous lemma we obtain the following estimate.
\begin{corollary}\label{biot_corollary}
Let $0<\sigma <1$.  Suppose $u$ is a weak solution to equation~\eqref{equation} in $\Omega_T \Supset 2Q$ satisfying~\eqref{VE} such that $u \in L^{2+\eps}(2Q)$, for some $\eps>0$.  Then there exists a constant $C=C(\eps, V_o)>0$ such that for every $(x,t) \in \sigma Q$ we have
\[
|u(x,t)-u_{2B}|\le \frac1{2\pi}\int_{B} \frac{|\omega(y,t)|}{|x-y|}  \, dy + \frac{C}{(1-\sigma)^3}\big[\|u(\cdot, t)-u_{2B}\|_{L^1(2B), avg}+\Gamma\big],
\]
where $\Gamma$ is as in~\eqref{gamma_lemma}. In particular, the constant $C$ is independent of the time variable $t$.\end{corollary}
\begin{proof}
Choose $k=\big((u_1)_{2B}, (u_2)_{2B}\big)$ in Lemma~\ref{vorticity_singular}.
\end{proof}

We will next use the above Biot-Savart law in place of a Sobolev embedding in a suitable Moser iteration scheme. This allows us to iteratively improve the integrability of the velocity field $u$. Observe that we only require some integrability of $\omega$ rather than the existence of full Sobolev gradients. 

\pagebreak

\section{Uniform integrability estimates for the weak solutions}

In this section we will prove Theorem~\ref{main_theorem}. The main steps of the proof are as follows:
\begin{enumerate}
\item[1.] use the vorticity estimates~\eqref{VE} to obtain an $L^q$ estimate for $\omega$ {\it uniformly in time}, where on the right hand side the integrability assumptions on $u$ and $\omega$ are used to control the different terms;
\item[2.] combine this uniform in time $L^q$ estimate with the Hardy-Littlewood-Sobolev estimate of Lemma~\ref{singular} and the Biot-Savart type estimate of Lemma~\ref{vorticity_singular} to obtain that $u \in L^{2+\rho}$, for some quantitatively determinable $\rho>\eps$;
\item[3.] plug the newly acquired estimate for $u$ into the first step in order to use H\"older's inequality with a higher power for $u$ and with a correspondingly lower power for $\omega$ so that we may choose $\alpha$ larger than previously to conclude an improved $L^q$ estimate for $\omega$, again uniformly in time;
\item[4.] iterate the process to obtain quantitative growth rate for the $L^p$-norm of $u$;
\item[5.] finally, the result follows from showing that the $L^p$-norms are exponentially summable.
\end{enumerate}
We proceed with the detailed proof.

\begin{proof}[Proof of Theorem~\ref{main_theorem}] According to the plan above, we divide the proof into five steps.

\medskip

\noindent \textbf{Step 1.}
Let  $r>0$ be small enough so that $Q_{2r} \Subset \Omega_T$. We apply the vorticity estimate~\eqref{VE} with increasing values of $\alpha$ for $0 < \alpha < 1$. Let $j \in \{0,1,2, \dots\}$ be fixed. We consider the problem in a sequence of cylinders $$Q_k \df=Q_{r_k}=  B_{r_k} \times I_{r_k}\df=B_k\times I_k$$ with
\[
r_k=(2-2^{k-j})r
\] 
for $k=0, 1, 2,\dots, j$.

Let $\zeta_k\in C^\infty(Q_k)$ be a sequence of cut-off functions such that $0 \le \zeta_k \le 1$, $\zeta_k=1$ in $Q_{k+1}$ and $\zeta_k = 0$ on $\partial_p Q_{k}$. Moreover, choose $\zeta_k$ such that
\begin{equation}\label{zeta_grad}
|\nabla \zeta_k| \le \frac{2^{j-k}C}{r} \quad \text{and}\quad |\partial_t \zeta_k| \le \frac{2^{j-k}C}{r}.
\end{equation}
Inserting this cut-off function into~\eqref{VE} yields
\begin{equation}\label{average_cacc}
\begin{split}
&\esssup_{t \in I_{k+1}} \dashint_{B_{k+1}} |\omega|^{1+\alpha} \, dx \\
&\le 2^{j-k}C\dashint_{Q_{k}} |\omega|^{1+\alpha} + |u||\omega|^{1+\alpha} \, d\mu \\
&\le 2^{j-k}C\dashint_{Q_{k}} |\omega|^{1+\alpha} \, d\mu + 2^{j-k}C\left(\dashint_{Q_{k}} |u|^q \, d\mu\right)^{\frac 1 q}\left(\dashint_{Q_{k}} |\omega|^{\frac{q(1+\alpha)}{q-1}} \, d\mu \right)^{\frac{q-1}{q}},
\end{split}
\end{equation}
for $2 < q \le \infty$ and for a uniform constant $C$. For each $q \in (2,\infty]$, we choose $\alpha \in (0,1]$ such that
\[
\frac{q(1+\alpha)}{q-1}=2.
\]
This choice yields $\alpha = 1- 2/q \in (0,1)$, and we obtain from~\eqref{average_cacc} that
\begin{equation}\label{omega_uniform}
\begin{split}
&\esssup_{t \in I_{k+1}} \dashint_{B_{k+1}} |\omega|^{2\l(1-\frac1q\r)} \, dx \\
&\le 2^{j-k}C\left(\dashint_{Q_k} |\omega|^{2} \, d\mu\right)^{1-\frac1q}+ 2^{j-k}C\left(\dashint_{Q_k} |u|^q \, d\mu\right)^{\frac1{q}}\left(\dashint_{Q_k} |\omega|^{2} \, d\mu \right)^{1-\frac1q}.
\end{split}
\end{equation}
Thus, if $u \in L^q(Q_k)$, we also have $\omega \in L^\infty(I_{k+1}; L^{2\l(1-\frac1q\r)}(B_{k+1}))$.
\medskip

\noindent \textbf{Step 2.}
By Lemma~\ref{vorticity_singular}, we further obtain for every $(x,t) \in B_{k+2} \times I_{k+2}$ that
  \begin{align*}
|u(x,t)|\le \frac1{2\pi}\int_{B_{k+1}} \frac{|\omega(y,t)|}{|x-y|}  \, dy + C8^{j-k}\Lambda_o,
 \end{align*}
 where we have denoted
 \[
 \Lambda_o=\Lambda_o(Q_o)\df=\sup_{t \in I_o} \dashint_{B_o} |u| \, dx+\l(4r^2\dashint_{Q_o} |\omega|^2\, d\mu+1\r)^{\frac12}\l(\dashint_{Q_o} |u|^{2+\eps}+1 \ d\mu\r)^{\frac1{2+\eps}}.
 \]
We estimate the right hand side above with the Hardy-Littlewood-Sobolev potential estimate of Lemma~\ref{singular} together with~\eqref{omega_uniform}. We will use the Lemma with $\omega \in L^{2\l(1-\frac1q\r)}$ and $q \in [2+\eps, \infty)$. Observe that the constant in the Lemma depends on $q$ and we have to carefully analyze the dependence. For $q \in [2+\eps, \infty)$, we obtain
\begin{equation}\label{u_uniform}
\begin{split}
 &\|u(\cdot, t)\|_{L^{ 2(q-1)}(B_{k+2})} \\
  &\le C\, q^\frac12 \|\omega(\cdot, t)\|_{L^{2\l(1-\frac1q\r)}(B_{k+1})}+C8^{j-k}  r^{\frac1{q-1}}\Lambda_o \\
  &\le C8^{j-k}q^\frac12r^{\frac{1}{q-1}}\l[r\l(\dashint_{Q_k} |\omega|^2 \, d\mu\r)^{\frac 12}\l[1+\l(\dashint_{Q_k} |u|^q \, d\mu\r)^\frac1q\r]^{\frac{q}{2(q-1)}}+\Lambda_o\r]
 \end{split}
 \end{equation}
uniformly for all $t  \in I_{k+2}$. 

Since the above estimate is uniform in time, averaging the integrals yields
 \begin{equation}\label{u_reverse}
 \begin{split}
 &\|u\|_{L^{ 2(q-1)}(Q_{k+2}), avg} \\
 &\le C8^{(j-k)}q^\frac12\big[r\|\omega\|_{L^2(Q_o), avg}[1+\|u\|_{L^q(Q_k), avg}]^{\frac{ q}{2(q-1)}}+\Lambda_o\big].
\end{split}
 \end{equation}
Observe that for $q>2$, we have 
\[
 2(q-1)>q.
\] 

\medskip

\noindent \textbf{Steps 3. and 4.} Starting from $u \in L^{2+\eps}$, for some $\eps>0$, we may iterate to eventually obtain that $u$ is integrable to an arbitrary high power. Indeed, we will iterate~\eqref{u_reverse} with increasing values of $q$. The constant will consequently blow up in the above estimate and we are not able to obtain local boundedness of $u$, as expected, since such a result is not true. Instead, we will show that our estimates are enough to show an exponential integrability estimate.

First, we will however, renumber the cylinders $Q_k$, since for technical reasons we will need to jump over $Q_{k+1}$ in the above estimate. We denote $l_k=2k$ and consider the subsequence $(Q_{l_k})$ instead of $(Q_k)$. For simplicity, we may, however, return notationally back to $(Q_k)$ and identify it with the subsequence $(Q_{l_k})$.

Choose $q_0=2+\eps$ as well as 
\[
q_{k+1}=2\l(q_k-1\r), 
\] 
for $0 \le k \le j-1$. This gives
\[
q_k=2^k\eps+2, \quad k=0,1,2, \dots.
\]
Plugging $q_k$ into~\eqref{u_reverse}, and taking in account the above renumbering of cylinders $Q_k$, yields
 \begin{align*}
 \|u\|_{L^{q_{k+1}}(Q_{k+1}), avg} &\le C^{(j-k)}q_k^\frac12\big[r\|\omega\|_{L^2(Q_o), avg}[1+\|u\|_{L^{q_k}(Q_k), avg}]^{\frac{ q_k}{q_{k+1}}}+\Lambda_o\big] \\
 &\le C^{(j-k)}q_k^\frac12\Lambda_o\l[\|u\|_{L^{q_{k}}(Q_{k}), avg}^{\frac{q_k}{q_{k+1}}}+1\r]
 \end{align*} 
 for every integer  $0\le k \le j-1$. Here we used the fact that
\[
r\|\omega\|_{L^2(Q_{o}), avg} \le \Lambda_o.
\] 

By enlarging the constant $C$ if necessary, we obtain by iteration that
 \begin{equation}\label{u_q}
 \begin{split}
 &\|u\|_{L^{q_{j}}(Q_j), avg}\\
  &\le Cq_{j-1}^\frac12\Lambda_o\l[\|u\|_{L^{q_{j-1}}(Q_{j-1}), avg}^{\frac{q_{j-1}}{q_{j}}}+1\r]\\
& \le Cq_{j-1}^\frac12\Lambda_o\l[C^2q_{j-2}^\frac12\Lambda_o\l(\|u\|_{L^{q_{j-2}}(Q_{j-2}), avg}^{\frac{q_{j-2}}{q_{j-1}}}+1\r)\r]^{\frac{q_{j-1}}{q_{j}}}\\
&\le Cq_{j-1}^\frac12\l[C^2q_{j-2}^\frac12\r]^{\frac{q_{j-1}}{q_j}}\Lambda_o^{1+{\frac{q_{j-1}}{q_j}}}\l[\|u\|_{L^{q_{j-2}}(Q_{j-2}), avg}^{\frac{q_{j-2}}{q_j}}+1\r]
 \\
&\ \  \vdots \\
&\le \prod_{k=1}^{j}\l[C^kq_{j-k}^\frac12\r]^{\frac{q_{j+1-k}}{q_j}}\Lambda_o^{q^*}\l[\|u\|_{L^{q_{o}}(Q_{o}), avg}^{\frac{q_{o}}{q_j}} +1\r] \\
&\le \Lambda_o^{2q^*}\prod_{k=1}^{j}\l[C^kq_{j-k}^\frac12\r]^{\frac{q_{j+1-k}}{q_j}},
 \end{split}
 \end{equation}
where
 \[
 q^*\df= \sum_{k=1}^{j}\frac{q_{k}}{q_j}=\frac1{2^j\eps+2}\sum_{k=1}^j \l[2^k\eps+2\r]=\frac{\l[2^j-1\r]\eps+ j}{2^{j-1}\eps+1} \le 4-\log_2\eps,
 \]
and
 \[
\frac {q_0}{q_j} =\frac {2+\eps}{2^j\eps+2}\le 1
 \]
for all $j \ge 1$ and $\eps >0$. 

For the right hand side of~\eqref{u_q} we have
\begin{align*}
\prod_{k=1}^j\l[C^kq_{j-k}^\frac12\r]^\frac{q_{j+1-k}}{q_j} \le q_j^{\frac {q^*}2}C^{\widehat q},
\end{align*}
where
\[
\widehat q = \sum_{k=1}^j k\frac{q_{j+1-k}}{q_j}=\sum_{k=1}^j k\frac{2^{j+1-k}\eps-2}{2^j\eps +2}\le \sum_{k=1}^j k2^{1-k} \le \sum_{k=1}^\infty k2^{1-k} < \infty,
\]
uniformly for all $j \ge 1$. Combining the above gives
\begin{equation}\label{q_j}
\|u\|_{L^{q_{j}}(Q_j), avg} \le Cq_j^{\frac{q^*}2}\Lambda_o^{2q^*},
\end{equation}
for a uniform constant $C$ independent of $j$.

\medskip

\noindent \textbf{Step 5.}
Let $k \ge 2$ and choose $j$  such that $$k \le q_j =2^j\eps +2 \le 2k.$$ 
Plugging this into~\eqref{q_j} yields
\[
\dashint_{Q_r} |u|^k \, d\mu\le C^kk^{\frac{kq^*}2} \Lambda_o^{2q^*k}
\]
for all $k \ge 2$. 
Using this to estimate the right hand side in~\eqref{u_uniform} gives
\begin{align*}
\esssup_{t \in I_{\frac r2}} \dashint_{B_\frac r2} |u(x, t)|^k \, dx &\le C^k k^{\frac k2}\Lambda_o^k\l[1+\dashint_{Q_r} |u|^k \, d\mu\r] \\
&\le  [C\Lambda_o^{c_o}]^kk^{c_ok},
\end{align*}
for $k \ge 2$, and
\[
c_o\df =2q^*+1.
\]
By using the previous estimate, together with H\"older's inequality, we obtain, for an integer $\gamma \ge 1$, that 
\begin{align*}
\dashint_{B_\frac r2} e^{\sqrt[\gamma]{|u(x,t)|}} \, dx  &\le \dashint_{B_\frac r2} \sum_{k=0}^\infty \frac{|u(x,t)|^{\frac k\gamma}}{k!} \, dx\\
&\le \sum_{k=0}^\infty \frac 1{k!} \dashint_{B_\frac r2}|u(x,t)|^\frac k\gamma \, dx\\
&\le C\l[\dashint_{B_\frac r2}|u(x,t)| \, dx+1\r]^\frac1\gamma +\sum_{k=2}^{\gamma^2-1} \frac {[C\Lambda_o^{c_o}]^\frac k\gamma k^{\frac{c_ok}{\gamma}}}{k!} \\
&\qquad+ \sum_{k=\gamma^2}^\infty \frac {[C\Lambda_o^{c_o}]^\frac k\gamma k^{\frac{c_ok}{\gamma}}}{k!} 
\end{align*}
uniformly for all $t \in I_\frac r2$. By Stirling's formula, there exists a uniform constant $C_1$ such that, by choosing $\gamma = C_1 c_o \Lambda_o$, we obtain for the tail
\[
\sum_{k=\gamma^2}^\infty \frac {[C\Lambda_o^{c_o}]^\frac k\gamma k^{\frac{c_ok}{\gamma}}}{k!} \le \sum_{k=\gamma^2}^\infty \frac {\l[\frac{C}{C_1}\r]^k k^k}{k!}  \le 1.
\]
We finally get
\[
\esssup_{t \in I_{\frac r2}}\dashint_{B_\frac r2} e^{\sqrt[\gamma]{|u(x,t)|}} \, dx \le C(\eps, \Lambda_o),
\]
where
\[
\Lambda_o=\Lambda_o(Q_o)=\esssup_{t \in I_o} \dashint_{B_o} |u| \, dx+C\l(r^2\dashint_{Q_o} |\omega|^2\, d\mu+1\r)^{\frac12}\l(\dashint_{Q_o} |u|^{2+\eps} \, d\mu+1 \r)^{\frac1{2+\eps}},
\]
for a uniform constant $C$. This finishes the proof.
\end{proof}

\begin{remark}\label{NS_remark}
The above proof of the main theorem only relies on the vorticity estimates~\eqref{VE} for the weak solutions of the Euler equation and on the Biot-Savart law of Lemma~\ref{vorticity_singular}. Therefore, the whole argument can also be completed for weak solutions of the incompressible Navier-Stokes system
\[
\partial_tu- \Delta u + (u\cdot\nabla)u =-\nabla p.
\]
Indeed, the diffusive viscosity term $\Delta u$ will only add a positive contribution on the left hand side of~\eqref{VE}. This can also be used to absorb the additional term appearing in the vorticity formulation
\[
\partial_t \omega - \Delta \omega + (u \cdot \nabla) \omega + (\omega \cdot \nabla)u=0
\]
 in higher dimensions. 

Consequently, it is an easy exercise to show that there exists a constant $C=C(\alpha_o)>0$ such that, for every $\alpha \ge \alpha_o >0$, the vorticity formally satisfies the energy estimate
\begin{equation}\label{NSVE}
\begin{split}
&\esssup_{0 \le t \le T} \int_{\Omega} |\omega|^{1+\alpha}\zeta^2(x,t) \, dx + \int_{0}^{T}\int_{\Omega} |\nabla |\omega|^{\frac{1+\alpha}{2}}|^2 \zeta^2 \, dx \, dt \\
&\le C\int_{0}^{T}\int_{\Omega} |\omega|^{1+\alpha}\l[|u||\nabla \zeta|+\chi_{\{d>2\}}|u|^2\zeta^2+\left(\frac{\partial \zeta}{\partial t}\right)_+\r] \, dx \, dt,
\end{split}
\end{equation}
for every non-negative test function $\zeta \in C^\infty(\Omega_T)$ vanishing on the parabolic boundary $\partial_p \Omega_T$.
Observe that here one requires the existence of weak gradients for $\omega$. A rigorous treatment removing this assumption requires an approximation argument, for which we refer to~\cite{Stru88}.

The above energy estimate now includes a term of the form $|u|^2|\omega|^{1+\alpha}$, in addition to $|u||\omega|^{1+\alpha}$. Therefore, for the higher-dimensional Navier-Stokes system the assumptions of Theorem~\ref{main_theorem} must be modified to $\omega \in L_{\rm loc}^d$ and $u \in L_{\rm loc}^{\kappa+\eps}$, for
\begin{equation}\label{kappa}
\kappa
=
\begin{cases}
2 \quad&\text{if} \ d=2, \\
\frac{2d}{d-2}\quad &\text{if} \ d \ge3.
\end{cases}
\end{equation}
The modifications, to extend Theorem~\ref{main_theorem} for the Navier-Stokes system, are straightforward and we omit the details, even though we exploit this in the next section, where we use our methodology to conclude the Serrin regularity theorem~\cite{Serr62}.
\end{remark}

\section{On the interior regularity for Navier-Stokes}

We will now comment on how our reasoning can be used to give a modified proof of the classical Serrin regularity theorem~\cite{Serr62} for weak solutions of the incompressible Navier-Stokes system
\begin{equation}\label{NS_equation}
\begin{cases}
\partial_t u -\Delta u+ u \cdot \nabla u &= \ -\nabla p, \\
\qquad \qquad\qquad  \nabla \cdot u &= \ \ 0,
\end{cases}
\end{equation}
in a space-time cylinder $\Omega_T \subset \Rn \times \R_+$. Moreover, we will also show how the assumptions can be slightly relaxed in two dimensions, while acknowledging that the original goal of Serrin was to consider the case $d \ge 3$; the well-posedness of the two-dimensional Navier-Stokes system had already been established by the seminal contributions of Leray, Hopf and Ladyzhenskaya ~\cite{Lera34, Hopf, Lady58}.

\begin{theorem}
Suppose
\[
u \in L_{\rm loc}^\infty(0, T; L_{\rm loc}^2(\Omega)) \cap L_{\rm loc}^q(0, T; L_{\rm loc}^s(\Omega))
\]
is a weak solution of the Navier-Stokes system~\eqref{NS_equation} such that $\omega \in L_{\rm loc}^2(\Omega_T)$ and
\begin{alignat}{2}
q,s&>2,   \quad&&\text{if}\ d=2,\notag\\   
\frac 2q + \frac d{s} &\le1 &&\text{if}\ d \ge 3.  \label{Serrin_condition}
\end{alignat}
Then $u \in C^\infty$ in the space variable with locally bounded derivatives.
\end{theorem}
\begin{remark}
The result is originally due to Serrin~\cite{Serr62}, where ~\eqref{Serrin_condition} is assumed also for $d=2$. In addition to providing a slightly modified argument for the result, we observe that this requirement can be relaxed in two dimensions to just assuming $q,s>2$. 

Observe that -- similarly to Serrin -- we do not assume the existence of Sobolev gradients, which makes it possible to apply our method also for the Euler equations, as we have seen. If one assumes that the solution is {\it a priori} in the energy space with full Sobolev gradients, then one may use the Sobolev embedding to obtain that a function satisfying our condition also satisfies the original Serrin condition in two dimensions. 
The borderline case of equality in~\eqref{Serrin_condition} is originally due to Fabes, Jones and Rivi\`ere~\cite{FabeJoneRivi72}. Later on, different arguments have been provided, for instance, by Struwe~\cite{Stru88}.

We also exclude the case $q=\infty$ and $s=d$, where important contributions have been made in the three-dimensional case, for instance, by Escauriaza, Seregin and \v{S}ver\'ak~\cite{EscaSereSver03}.

\end{remark}

\begin{proof}

The Navier-Stokes and the Euler equations have different scalings and for this reason we redefine the space-time cylinders by setting
\[
\sigma Q=\sigma B\times \sigma I =B(0,\sigma r) \times (-\sigma r^2, 0),
\]
for $\sigma >0$. 
Let $0< \delta <1$ and choose a standard cut-off function $\zeta \in C^\infty((1+\delta)Q)$ such that $\zeta = 0$ on $\partial_p (1+\delta)Q$ with $\zeta = 1$ in $Q$, and 
\[
|\nabla \zeta| \le \frac C{\delta r}\quad \text{as well as} \quad \l|\frac{\partial \zeta}{\partial t}\r| \le \frac C {(\delta r)^2}.
\]
We may now apply the energy estimate~\eqref{NSVE} for the vorticity formulation of the Navier-Stokes equations. 

Let $0<r <1$ be so small that 
\[
Q_{3r}=B_{3r}(0) \times \big(-3r^2,0\big) \subset B_{3r}(0) \times (-3r,0)\df=U \Subset \Omega \times (0, T).
\]
Assume first that $u$ is in $L^p$, for $p$ large enough. We will later use Theorem~\ref{main_theorem} and the subsequent Remark~\ref{NS_remark} to prove this. We obtain from the parabolic Sobolev embedding~\eqref{parabolic_Sobolev} and the energy estimate~\eqref{NSVE} that
\begin{equation}\label{omega_reverse}
\begin{split}
&\l(\dashint_{I}\l(\dashint_{B} \big(|\omega|^{\frac{1+\alpha}2}\big)^{s^*}\, dx\r)^{\frac{q^*}{s^*}} \, dt \r)^{\frac{2}{q^*}}\\
&\le C \l(\dashint_{(1+\delta)I}\l(\dashint_{(1+\delta)B} \big(|\omega|^{\frac{1+\alpha}2}\zeta\big)^{s^*}\, dx\r)^{\frac{q^*}{s^*}} \, dt \r)^{\frac{2}{q^*}}\\
&\le C \l(\esssup_{ (1+\delta)I} \dashint_{ (1+\delta)B} |\omega|^{1+\alpha}\zeta^2(x,t) \, dx + r^2\dashint_{ (1+\delta)Q} \big|\nabla \big(|\omega|^{\frac{1+\alpha}{2}}\zeta\big)\big|^2 \, d\mu\r) \\
&\le Cr^2\dashint_{ (1+\delta)Q}|\omega|^{1+\alpha}\l[|u||\nabla \zeta|+\chi_{\{d>2\}}|u|^2\zeta^2+\left(\frac{\partial \zeta}{\partial t}\right)_+\r] \, d\mu\\
&\le \frac {C\big(1+r^2\|u\|_{L^p, avg}\big)}{\delta^2}\dashint_{(1+\delta)Q} |\omega|^{(1+\alpha)\frac p{p-2}}\, d\mu,
\end{split}
\end{equation}
for every $p >1$ and for every $(1+\delta)Q \Subset \Omega_T$, as well as for every $2 < s^*,q^* <\infty$ with
\begin{equation}\label{star}
\frac d{s^*} +\frac2{q^*} \ge \frac d2.
\end{equation}
Now we may choose
\[
q^*=s^*=2\l(1+\frac2d\r)
\]
and $p$ large enough to guarantee that $$\frac {s^*}2 > \frac p{p-2}.$$ This allows us to conclude, by a standard Moser iteration, that \[
\esssup_{ Qr} |\omega| \le C\l(\dashint_{Q_{2r}} |\omega|^2 \, d\mu\r)^{1/2},
\]
where the constant $C$  depends on the locally uniform bound of $\|u\|_{L^p}$. This is achieved by iterating~\eqref{omega_reverse} for increasing values of $\alpha$, starting with $\alpha=1$. The theorem now follows as in Serrin~\cite{Serr62} from the standard iterative argument of using the convolution representation of $\omega$ in terms of $u$. 

It remains to show the $L^p$-boundedness of $u$ in $ Q_{2r}$. For $d=2$, we immediately obtain, from Theorem~\ref{main_theorem} and Remark~\ref{NS_remark}, that if
\[
u \in L^\infty(I_{3 r}; L^1(B_{3r})) \cap L^{2+\eps}(U) \quad \text{for some}\  \eps>0,
\] 
then $u \in L^p(U)$, for all $p >1$, which finishes the proof. For higher dimensions, we need to show that the assumptions of Theorem~\ref{main_theorem} hold. This can be done in a manner similar to ~\cite{Stru88}, but instead of proving that $\omega \in L^{\infty}(2I;L^{d+\eps}(2B))$, it is enough to show that
\begin{equation}\label{omega_integrability}
\omega \in L^d(2Q) \cap L^\infty(2I; L^{2+\eps}(2B))
\end{equation}
for some $\eps >0$. Indeed, by Lemma~\ref{singular}, we have that if $\omega \in L^\infty(I; L^{2+\eps}(B))$, we directly obtain from the Biot-Savart law of Lemma~\ref{vorticity_singular} that $u \in L^{\kappa+\rho}( Q(2r))$, for $\kappa$ as in~\eqref{kappa}, and $\rho=\rho(\eps)>0$. This is then enough to conclude the proof as described in the Remark \ref{NS_remark}.

We proceed similarly to~\cite{Stru88}. By using the fact that $u \in L_{\rm loc}^q(L_{\rm loc}^s)$ we may use H\"older's inequality to control the term containing $|\omega|^{1+\alpha}|u|^2$, on the right hand side of~\eqref{omega_reverse}, as
\begin{equation}\label{absorb}
\||\omega|^{1+\alpha}|u|^2\zeta^2\|_{L^1} \le \|\omega\zeta^{\frac{2}{1+\alpha}}\|_{L^{\frac{(1+\alpha)q_o}{2}}\big(L^{\frac{(1+\alpha)s_o}{2}}\big)}^{1+\alpha}\|u\|_{L^{q}(L^{s})(\spt \zeta)}^{2},
\end{equation}
where
\[
s_o \df= \frac{2s}{s-2}\quad \text{and} \quad  q_o \df= \frac{2q}{q-2}.
\]
In order to improve the integrability of $\omega$ with the above estimate we need to guarantee that 
\begin{equation}\label{star_improvement}
s^* \ge s_o\quad \text{and} \quad q^* \ge q_o.
\end{equation}
In this case we may absorb the term obtained in~\eqref{absorb} to the left hand side of~\eqref{omega_reverse} by choosing the support of the test function small enough. This is possible as long as the pair $(q_o, s_o)$  satisfies the condition ~\eqref{star}, i.e.,
\[
\frac d2 \le\frac d{s_o} +\frac2{q_o} = 1+\frac d2 -\frac ds -\frac 2q.
\]
This yields the Serrin condition~\eqref{Serrin_condition} for $s$ and $q$. Finally, the result follows from iterating~\eqref{omega_reverse} for increasing values of $\alpha$. For the details we refer to~\cite{Stru88}, where the only difference is that we may conclude the iteration already after obtaining~\eqref{omega_integrability}. 

\end{proof}

\noindent {\bf Acknowledgments.} This work initiated while the authors were visiting the Mittag-Leffler Institute during the program "Evolutionary Problems" and was concluded while J.S. visited the University of Coimbra; we thank both institutions for the kind hospitality. J.S. was supported by the Academy of Finland grant 259363 and by a V\"ais\"al\"a Foundation travel grant. J.M.U. was partially supported by the Centre for Mathematics of the University of Coimbra -- UID/MAT/00324/2013, funded by the Portuguese Government through FCT/MCTES and co-funded by the European Regional Development Fund through the Partnership Agreement PT2020.

\end{document}